\documentclass{daj}
\usepackage{amsfonts,amsthm,amsmath,amssymb,mathtools,enumerate,hyperref}

\newcommand{\A}{\mathbf{A}}
\newcommand{\Q}{\mathbf{Q}}

\newcommand{\C}{\mathbf{C}}

\newcommand{\Z}{\mathbf{Z}}

\newcommand{\Fp}{\mathbf{F}_{\!p}}

\newcommand{\F}{\mathbf{F}}

\newcommand{\GL}{{\mathbf{GL}}}
\newcommand{\SL}{{\mathbf{SL}}}
\newcommand{\PGL}{{\mathbf{PGL}}}
\newcommand{\PSL}{{\mathbf{PSL}}}

\renewcommand{\O}{\mathcal{O}}

\newcommand{\p}{\mathfrak p}
\newcommand{\q}{\mathfrak q}

\renewcommand{\Re}{\operatorname{Re}}

\newcommand{\cP}{\mathcal P}

\def\Gal{\operatorname{Gal}}
\def\Aut{\operatorname{Aut}}
\def\Hom{\operatorname{Hom}}

\def\det{\operatorname{det}}

\def\hairspace{\kern .05em}

\newcommand{\smat}[4]{\bigl[\begin{smallmatrix}#1&#2\\ #3&#4\end{smallmatrix}\bigr]}
\newcommand{\mr}[1]{(\href{https://mathscinet.ams.org/mathscinet-getitem?mr=#1}{MR#1})}

\numberwithin{equation}{section}

\newtheorem{theorem}{Theorem}[section]

\newtheorem{corollary}[theorem]{Corollary}
\newtheorem{lemma}[theorem]{Lemma}
\newtheorem{proposition}[theorem]{Proposition}

\theoremstyle{definition}
\newtheorem{definition}[theorem]{Definition}

\newtheorem{remark}[theorem]{Remark}
\newtheorem{question}[theorem]{Question}

\dajAUTHORdetails{%
  title = {Stronger arithmetic equivalence}, 
  author = {Andrew V. Sutherland},
  plaintextauthor = {Andrew V. Sutherland},
}   

\dajEDITORdetails{%
   year={2021},
   number={23},
   received={28 April 2021},   
   published={15 November 2021},  
   doi={10.19086/da.29452},       
}   

\begin{document}

\begin{frontmatter}[classification=text]

\title{Stronger arithmetic equivalence} 

\author[avs]{Andrew V. Sutherland\thanks{Supported by NSF grant DMS-1522526 and Simons Foundation grant 550033.}}

\begin{abstract}
Motivated by a recent result of Prasad, we consider three stronger notions of arithmetic equivalence: \emph{local integral equivalence}, \emph{integral equivalence}, and \emph{solvable equivalence}.
In addition to having the same Dedekind zeta function (the usual notion of arithmetic equivalence), number fields that are equivalent in any of these stronger senses must have the same class number, and solvable equivalence forces an isomorphism of adele rings.  Until recently the only nontrivial example of integral and solvable equivalence arose from a group-theoretic construction of Scott that was exploited by Prasad.  Here we provide infinitely many distinct examples of solvable equivalence, including a family that contains Scott's construction as well as an explicit example of degree 96.  We also construct examples that address questions of Scott, and of Guralnick and Weiss, and shed some light on a question of Prasad.
\end{abstract}
\end{frontmatter}

\section{Introduction}

Number fields that have the same Dedekind zeta function are said to be \emph{arithmetically equivalent}.  Arithmetically equivalent number fields need not be isomorphic, but they necessarily have the same normal closure and share many arithmetic invariants.  The first nontrivial example of arithmetically equivalent number fields was given by Gassmann \cite{Gassmann}, who showed that all such examples arise from a simple group-theoretic construction, a \emph{Gassmann triple} $(G,H_1,H_2)$ of finite groups in which $H_1$ and~$H_2$ are subgroups of $G$ that intersect every $G$-conjugacy class with the same cardinality; see Proposition~\ref{prop:gassmann} for several equivalent definitions.
Gassmann proved that number fields $K_1$, $K_2$ with Galois closure~$L$ are arithmetically equivalent if and only if $(\Gal(L/\Q),\Gal(L/K_1),\Gal(L/K_2))$ is a Gassmann triple.
The number fields $K_1$ and $K_2$ are isomorphic if and only if $\Gal(L/K_1)$ and $\Gal(L/K_2)$ are conjugate in $\Gal(L/\Q)$; we are thus interested in \emph{nontrivial Gassmann triples} $(G,H_1,H_2)$, those in which $H_1$ and~$H_2$ are nonconjugate subgroups of $G$.

Gassmann triples $(G,H_1,H_2)$ naturally arise in many other settings, most notably in the construction of isospectral manifolds.  As shown by Sunada \cite{Sunada}, if $\pi\colon M\to M_0$ is a normal finite Riemannian covering with transformation group $G$, the quotient manifolds $M/H_1$ and $M/H_2$ are \emph{isospectral}: they have the same sequence of Laplacian eigenvalues. Unlike the number field case, $M_1$ and $M_2$ may be isometric even when $H_1$ and $H_2$ are nonconjugate, but if $H_1$ and $H_2$ are nonisomorphic and $M$ is the universal covering of~$M_0$, then $M/H_1$ and $M/H_2$ have nonisomorphic fundamental groups~$H_1$ and $H_2$ and cannot be isometric; see \cite[\S 4.2]{Shafarevich} and \cite[Corollary\ 1]{Sunada}.  A consequence of this result is that there are infinitely many distinct ways in which one cannot ``hear the shape of a drum'' \cite{GWW,Kac,Milnor}.
A similar result holds in algebraic geometry: if $X$ is a projective curve over a number field $K$ and  $(G,H_1,H_2)$ is a Gassmann triple with $G\subseteq\Aut(X)$, then the Jacobians of the quotient curves $X/H_1$ and $X/H_2$ are isogenous over~$K$, as proved by Prasad and Rajan in \cite{PR}.  As shown in \cite{AKMS}, this result can be generalized to \'etale Galois covers of $K$-varieties.
There is also a discrete analog to Sunada's theorem in which one considers a finite graph~$\Gamma$ with automorphism group $G$: Gassmann triples $(G,H_1,H_2)$ can be used to construct nonisomorphic isospectral graphs $\Gamma/H_1$ and $\Gamma/H_2$, subject to conditions on $H_1$ and~$H_2$; see~\cite{HH}.  An introduction to the topics of arithmetic equivalence and isospectrality can be found in \cite{SutherlandNotes}.

Subgroups $H_1,H_2$ of $G$ form a Gassmann triple $(G,H_1,H_2)$ if and only if the permutation modules $\Q[H_1\backslash G]$ and $\Q[H_2\backslash G]$ given by the $G$-action on right cosets are isomorphic as $\Q[G]$-modules. We then say that $H_1$ and $H_2$ are \emph{rationally equivalent}, and if $\Z[H_1\backslash G]$ and $\Z[H_2\backslash G]$ are isomorphic as $\Z[G]$-modules, we say that $H_1$ and $H_2$ are \emph{integrally equivalent}. Prasad calls $(G,H_1,H_2)$ a \emph{refined Gassmann triple} when $H_1,H_2\le G$ are integrally equivalent, and shows that if $G$ is the Galois group of a Galois number field $L/\Q$ then the fixed fields $K_1:=L^{H_1}$ and $L_2:=K_2^{H_2}$ not only have the same Dedekind zeta function, they must also have isomorphic idele groups (and in particular, isomorphic class groups); see \cite[Theorem~2]{Prasad}. The first (and so far only) nontrivial example of a refined Gassmann triple was constructed by Scott~\cite{Scott92} more than thirty years ago. Prasad notes that this example can be realized by number fields, and that such number fields not only have the same Dedekind zeta function and isomorphic idele groups, they have isomorphic rings of adeles \cite[Theorem~3]{Prasad}, and are thus \emph{locally isomorphic}, meaning their local algebras are isomorphic at every place (see Theorem \ref{thm:localiso}).  Thus even when taken in aggregate these invariants are not enough to guarantee an isomorphism of number fields.\footnote{One can attach auxiliary $L$-functions to a number field that in combination with the Dedekind zeta function ensure an isomorphism of number fields whenever all of these $L$-functions coincide~\cite{CdSXMS,Smit}; for function fields see \cite{BV,CKvdZ,Solomatin}.}

As noted by Prasad, Scott's example is essentially the only nontrivial example of integral equivalence currently known \cite[Remark~1]{Prasad}.  It is not clear whether the particular feature of Scott's refined Gassmann triple that allowed Prasad to prove an isomorphism of adele rings is necessarily enjoyed by others (assuming there are any). Whether Scott's example is a singular special case or just the most accessible example of a general phenomenon remains an open question.

In this article we consider two alternative strengthenings of the notion of arithmetic equivalence: \emph{local integral equivalence} and \emph{solvable equivalence}.  The latter  implies the former and is sufficient to prove equality of the number field invariants considered by Prasad, notably including local isomorphism, which is not obviously implied by integral equivalence (indeed, we show that it is not implied by the similar but weaker notion of local integral equivalence).

An attractive feature of both local integral equivalence and solvable equivalence is that they are much easier conditions to check than integral equivalence; see  Propositions~\ref{prop:pred} and \ref{prop:localequiv}, and Definition~\ref{def:solvequiv}.
In this article we provide infinitely many nontrivial examples of solvably equivalent triples $(G,H_1,H_2)$ that can be realized as Galois groups of number fields (see Theorem~\ref{thm:solvequiv}), and we prove that
\begin{itemize}
\setlength{\itemsep}{-4pt}
\item locally integrally equivalent subgroups need not be isomorphic (see \S\ref{sec:locnoniso});
\item locally integrally equivalent number fields need not be locally isomorphic (see \S\ref{sec:locnonlociso});
\item locally integrally equivalent subgroups need not be integrally equivalent (see \S\ref{sec:d32});
\item solvably equivalent subgroups need not be integrally equivalent (see \S\ref{sec:d96}).
\end{itemize}

We construct an explicit example of locally integrally equivalent number fields of degree 32 arising from a triple $(G,H_1,H_2)$ with $H_1\not\simeq H_2$, and an explicit example of solvably equivalent number fields of degree 96 that are not integrally equivalent.  Thus solvable equivalence does not imply integral equivalence; we leave open the question of whether integral equivalence implies solvable equivalence.

The example in \S\ref{sec:locnoniso} negatively answers a question of Guralnick and Weiss \cite[Question 2.11]{GWe} and is relevant to the question of Prasad \cite[Question 1]{Prasad} as to whether integrally equivalent subgroups must be isomorphic, since it shows that locally integrally equivalent subgroups need not be. The solvably equivalent subgroups we construct are all isomorphic, which leads to the question of whether solvably equivalent subgroups are necessarily isomorphic.  This question is perhaps more accessible than Prasad's question, since the only example of integral equivalence currently known arises in a setting where rational equivalence is already enough to force isomorphism (rationally equivalent subgroups of $\GL_2(\Fp)$, $\SL_2(\Fp)$, $\PSL_2(\Fp)$ must be isomorphic, see \cite[Question~1]{Prasad} and \cite[Remark~3.7]{Sutherland}).  The example of solvable equivalence given in \S\ref{sec:d96} does not arise in this setting, and one can find many others.

The example in \S\ref{sec:locnonlociso} refines an answer to a question of Stuart and Perlis \cite[\S 4]{PerlisStuart} given by Mantilla-Soler \cite[Theorem 3.7]{Mantilla-Soler} by showing that the sum of the ramification indices above a given prime in arithmetically equivalent number fields need not coincide even when their products do (as they must for number fields that are locally integrally equivalent; see Proposition~\ref{prop:localequiv}).

The examples in \S\ref{sec:d32}, \S\ref{sec:d96} negatively answer a question of Guralnick and Weiss \cite[Question~2.10]{GWe} as to whether local integral equivalence implies integral equivalence (this question appears to have also been addressed in the thesis of D. Hahn in the case of solvable groups \cite{Roggenkamp}).  The example in \S\ref{sec:d96} also addresses a question of Scott \cite[Remark~4.3]{Scott92} regarding low rank permutation modules.

\section{Background and preparation}
In this section we recall background material, set notation, and summarize some of the results we will use.  The material in this section is well known to experts, but we  provide short proofs in cases where we were unable to find a suitable reference (a few of these results seem to be folklore).

Let~$G$ be a finite group.
For each subgroup $H\le G$ we use $[H\backslash G]$ to denote the transitive $G$-set given by the (right) action of $G$ on (right) cosets of $H$; this action is faithful if and only if the intersection of all the $G$-conjugates of $H$ (its \emph{normal core} in $G$) is the trivial group.  We use $\chi_H\colon G\to \Z$ to denote the permutation character $g\mapsto \#[H\backslash G]^g$ that sends $g$ to the number of $H$-cosets it fixes (the induced character~$1_H^G$), and note that $\chi_H(g)\ne 0$ if and only if $g$ is conjugate to an element of $H$.

We extend~$\chi_H$ to subgroups $K$ of $G$ by defining $\chi_H(K)\coloneqq \#[H\backslash G]^K$ (the \emph{mark} of $K$ on $[H\backslash G]$), so that $\chi_H(\langle g\rangle) = \chi_H(g)$; equivalently, $\chi_H(K)$ is the number of singleton fibers in the map $[H\backslash G]\to [H\backslash G/K]$ defined by $Hg\mapsto HgK$.
For $g\in G$ we have $HgK=Hg$ if and only if $gKg^{-1}\subseteq H$, and it follows that

\begin{equation}\label{eq:chiHK}
\chi_H(K) = \frac{\#\bigl\{g\in G:gKg^{-1}\le H\bigr\}}{\#H} = \frac{\#N_G(K)}{\#H} \#\bigl\{ gKg^{-1}\le H:g\in G\bigr\},
\end{equation}
where $N_G(K)$ denotes the normalizer of $K$ in $G$.

\begin{lemma}\label{lem:charconj}
Let $H$ and $K$ be subgroups of a finite group $G$.  The integer $\chi_H(K)$ depends only on the $G$-conjugacy classes of $H$ and $K$ and the function $\chi_H$ depends only on the $G$-conjugacy class of~$H$.
\end{lemma}
\begin{proof}
Replacing either $H$ or $K$ with a $G$-conjugate does not change the RHS of \eqref{eq:chiHK}.
\end{proof}

If $\mathcal P$ is a class of groups (e.g. cyclic groups or solvable groups), we call its elements \emph{$\cP$-groups}, and refer to the subgroups of a group $G$ that lie in $\mathcal P$ as its \emph{$\cP$-subgroups}.  We say that $\cP$ is \emph{subgroup-closed} if it contains all subgroups of its elements.
A function that maps subgroups of $G$ to subgroups of $G$ is \emph{$G$-class preserving} if it maps subgroups to $G$-conjugates.

\begin{proposition}\label{prop:pred}
Let $G$ be a finite group and let $\mathcal P$ be a subgroup-closed class of groups.
For any two subgroups $H_1,H_2\le G$ the following are equivalent:
\begin{enumerate}[{\rm (i)}]
\setlength{\itemsep}{0pt}
\item There is $G$-class preserving bijection between the sets of $\cP$-subgroups of $H_1$ and $H_2$;
\item $\#\{gKg^{-1}\le H_1:g\in G\} = \#\{gKg^{-1}\le H_2:g\in G\}$ for every $\cP$-subgroup $K$ of $G$;
\item $\chi_{H_1}(K)=\chi_{H_2}(K)$ for every $\cP$-subgroup $K$ of $G$;
\item $\chi_{H_1}(K)=\chi_{H_2}(K)$ for every $\cP$-subgroup $K$ of $H_1$ or $H_2$;
\item the $G$-sets $[H_1\backslash G]$ and $[H_2\backslash G]$ are isomorphic as $K$-sets for every $\cP$-subgroup $K$ of $G$.
\item the $G$-sets $[H_1\backslash G]$ and $[H_2\backslash G]$ are isomorphic as $K$-sets for every $\cP$-subgroup $K$ of $H_1$ or $H_1$.
\end{enumerate}
\end{proposition}
\begin{proof}
(i) $\Leftrightarrow$ (ii) is immediate, (ii) $\Leftrightarrow$ (iii) follows from \eqref{eq:chiHK}, (iii) $\Leftrightarrow$ (iv) follows from the fact that $\chi_{H_i}(K)=0$ if $K$ is not conjugate to a subgroup of~$H_i$, and the implications (v) $\Rightarrow$ (vi) $\Rightarrow$ (iv) are clear; it thus suffices to show (iii) $\Rightarrow$ (v).

Let $K$ be a $\cP$-group and consider the $K$-sets $X_1:=[H_1\backslash G]$ and $X_2=[H_2\backslash G]$.
We have $\chi_{H_1}(k)=\chi_{H_2}(k)$ for all $k\in K$, thus $\#X_1=\chi_{H_i}(1)=\#X_2$, and $X_1$ and $X_2$ both have $\frac{1}{\#K}\sum_k \chi_{H_i}(k)$ orbits.  The stabilizer $H\le K$ of any element $x$ in a $K$-orbit $X$ of $X_1$ or $X_2$ is a $\cP$-group, and $\chi_{H_1}(H)=\chi_{H_2}(H)$ implies that the $K$-orbits of $X_1$ and $X_2$ can be put in a bijection that preserves conjugacy classes of stabilizers (and thus preserves cardinalities).  If $H$ is the stabilizer of an element of a $K$-orbit $X$ in $X_1$ or $X_2$, then $X$ is isomorphic to the $K$-set $[H\backslash K]$.
It follows that $X_1$ and $X_2$ are isomorphic $K$-sets.
\end{proof}

We call subgroups $H_1,H_2\le G$ that satisfy the equivalent properties of Proposition~\ref{prop:pred} \emph{$\cP$-equivalent}, and this defines an equivalence relation on the subgroups of $G$. A necessary condition for $\cP$-equivalence is that $H_1$ and $H_2$ must have the same \emph{$\cP$-statistics}, meaning that they contain the same number of $\cP$-subgroups in every isomorphism class of groups.  When $\cP$ is the class of cyclic groups this amounts to having the same \emph{order statistics} (numbers of elements of each order).

\begin{remark}
Condition (iv) of Proposition~\ref{prop:pred} provides an efficient method for testing whether two subgroups of $G$ are $\cP$-equivalent.
Conditions (ii) and (iii) can both be used to efficiently partition the set of subgroups of $G$ into $\cP$-equivalence classes without the need for pairwise testing; conjugate subgroups lie in the same equivalence class, so it suffices to work with a set of conjugacy class representatives.
\end{remark}

\begin{lemma}\label{lemma:regiso}
Let $K,H_1,H_2$ be finite groups with $|H_1|=|H_2|=n$, let $\phi_1\colon K\to H_1$ and $\phi_2\colon K\to H_2$ be injective group homomorphisms, and let $\rho_1\colon H_1\to S_n$ and $\rho_2\colon H_2\to S_n$ be left regular representations of $H_1$ and $H_2$.  Then $K_1\coloneqq \rho_1(\phi_1(H_0))$ and $K_2\coloneqq\rho_2(\phi_2(H_0))$ are conjugate subgroups of $S_n$.
\end{lemma}
\begin{proof}
The $K$-sets $X_1=X_2=\{1,\ldots,n\}$ given by the actions of $K_1\coloneqq\rho_1(\phi_1(H_0))$ and $K_2\coloneqq \rho_2(\phi_2(H_0))$ are free $K$-sets with the same number of orbits, hence isomorphic.  There is thus a $K$-equivariant map $\sigma\colon X_1\to X_2$ with $\sigma\in S_n$ satisfying $ \rho_1(\phi_1(k))\sigma = \sigma\rho_2(\phi_2(k))$ for $k\in K$, and $\sigma^{-1}K_1\sigma=K_2$.
\end{proof}

\begin{corollary}\label{cor:regiso}
Let $\cP$ be a subgroup-closed class of groups.
Finite groups $H_1$ and $H_2$ of the same order can be embedded as $\cP$-equivalent subgroups of some group $G$ if and only if they have the same $\cP$-statistics.
\end{corollary}
\begin{proof}
The necessity of having the same $\cP$-statistics is obvious, and Lemma~\ref{lemma:regiso} proves sufficiency.
\end{proof}

For any integral domain $R$, we use $R[H\backslash G]$ to denote the corresponding permutation module; this is the free $R$-module with basis $[H\backslash G]$ equipped with the $R$-linear extension of the $G$-action on $[H\backslash G]$; we thus view $R[H\backslash G]$ as a (right) $R[G]$-module.

If $H_1,H_2\le G$ have the same index $n$, after ordering the $G$-sets $[H_1\backslash G]$ and $[H_2\backslash G]$, we may uniquely identify each $R[G]$-module homomorphism $R[H_1\backslash G]\to R[H_2\backslash G]$ with a matrix $M\in R^{n\times n}$ whose determinant $\det M$ does not depend on our choices.
If $\rho_1,\rho_2\colon G\to S_n$ are the permutation representations of~$G$ acting on $\{1,\ldots,n\}$ via our chosen orderings of $[H_1\backslash G]$ and $[H_2\backslash G]$, respectively, then the matrices $M\in R^{n\times n}$ that correspond to elements of $\Hom_{R[G]}(R[H_1\backslash G],R[H_2\backslash G])$ are precisely those that are fixed by the diagonal action of $\rho_1\times \rho_2$ on matrix entries; in other words, the entries of $M$ must satisfy
\[
M_{ij} = M_{\rho_1(g)(i),\rho_2(g)(j)}\qquad(\text{for all }g\in G).
\]
We define
\[
d(H_1,H_2):=\gcd \left\{\det M:M\in \Hom_{\Z[G]}(\Z[H_1\backslash G],\Z[H_2\backslash G])\right\},
\]
and extend this definition to all subgroups of $G$ by defining $d(H_1,H_2)=0$ whenever $\#H_1\ne \#H_2$.

We now give several equivalent conditions for subgroups to be $\cP$-equivalent when $\cP$ is the class of cyclic groups.

\begin{proposition}\label{prop:gassmann}
Let $G$ be a finite group.  For all subgroups $H_1$ and $H_2$ of $G$ the following are equivalent:
\begin{enumerate}[{\rm (i)}]
\setlength{\itemsep}{-1pt}
\item There is a bijection of sets $H_1\leftrightarrow H_2$ that preserves $G$-conjugacy;
\item $\#(H_1\cap C) = \#(H_2\cap C)$ for every conjugacy class $C$ of $G$;
\item $\chi_{H_1}(K)=\chi_{H_2}(K)$ for every cyclic $K\le G$;
\item The $G$-sets $[H_1\backslash G]$ and $[H_2\backslash G]$ are isomorphic as $K$-sets for every cyclic $K\le G$;
\item $\Q[H_1\backslash G]\simeq\Q[H_2\backslash G]$;
\item $d(H_1,H_2)\ne 0$.
\end{enumerate}
\end{proposition}
\begin{proof}
The equivalence of (i) and (ii) is immediate.
The equivalence of (ii) and (iii) follows from the formula $\chi_{H_i}(g)\#H_i=\#(H_i\cap C(g))\#Z(g)$,
where $C(g)$ is the conjugacy class of $g$ and $Z(g)$ is its centralizer (in $G$); see \cite[Eq.\,8]{Perlis78}.
The equivalence of (iii) and (iv) follows from applying Proposition~\ref{prop:pred} to the class of cyclic groups.
For the equivalence of (iii) and (v), note that $\dim_\Q(\Q[H_i\backslash G]^K)=\chi_{H_i}(K)$ for cyclic $K\le G$ and then apply the corollary to \cite[Theorem~30]{Serre}.
Clearing the denominators in $M\in \Hom_{\Q[G]}(\Q[H_1\backslash G],\Q[H_2\backslash G])$ shows the equivalence of (v) and (vi).
\end{proof}

\begin{remark}
The condition $K\le G$ in (iii), (iv) can be replaced by ``$K\le H_1$ or $K\le H_2$'' via Lemma 2.2.
\end{remark}

\begin{definition}
Subgroups $H_1$ and $H_2$ of a finite group $G$ that satisfy the equivalent conditions of Proposition~\ref{prop:gassmann} are said to be \emph{rationally equivalent} (or \emph{Gassmann equivalent}).
\end{definition}

A triple of groups $(G,H_1,H_2)$ with $H_1,H_2\le G$ rationally equivalent is called a \emph{Gassmann triple} \cite{Gassmann}.
By Proposition~\ref{prop:gassmann}, rational equivalence defines an equivalence relation on the subgroups of $G$.  Conjugate subgroups of $G$ are necessarily rational equivalent, so we may view this as an equivalence relation on conjugacy classes of subgroups.
Rational equivalence classes may be arbitrarily large \cite{Komatsu84}.

We are interested in the nontrivial rational equivalence classes, those which contain nonconjugate but rationally equivalent subgroups $H_1,H_2\le G$.
Equivalently, we are interested in the cases where $[H_1\backslash G]\not\simeq [H_2\backslash G]$ as $G$-sets, but $\Q[H_1\backslash G]\simeq\Q[H_2\backslash G]$ as $G[\Q]$-modules.
Standard examples include the subgroups $H_1\coloneqq \bigl\{\smat{1}{*}{0}{*}\in \GL_2(\F_p)\bigr\}$ and $H_2\coloneqq \bigl\{\smat{1}{0}{*}{*}\in \GL_2(\F_p)\bigr\}$ of $G\coloneqq \GL_2(\F_p)$, where $p$ is an odd prime \cite{deSmit98}, and similar examples in $\GL_n(\F_p)$ for $n>2$ and any prime $p$.
In these examples the subgroups $H_1$ and $H_2$ are not $G$-conjugate, but transposition gives a bijection $H_1\leftrightarrow H_2$ that preserves $G$-conjugacy.
The smallest example occurs for the group $G$ with GAP identifier $\langle 32,43\rangle$, which contains two nonconjugate rationally equivalent subgroups $H_1$ and $H_2$ isomorphic to the Klein $4$-group.\footnote{A GAP identifier $\langle m,n\rangle$ denotes the isomorphism class of an abstract group of order $m$; the positive integer $n$ is an ordinal that distinguishes distinct isomorphism classes of groups of order $m$. For $m\le 2000$ not equal to 1024 explicit presentations of these groups can be found in the small groups database \cite{BEO}, which is available in both GAP \cite{GAP} and Magma \cite{Magma}.}

\begin{remark}\label{rem:noniso}
Rationally equivalent subgroups necessarily have the same order but need not be isomorphic.  The smallest example of a Gassmann triple $(G,H_1,H_2)$ with $H_1\not\simeq H_2$ arises for $G\simeq \langle 384,5755\rangle$ with subgroups $H_1\simeq \langle 16,3\rangle$ and $H_2\simeq \langle 16,10\rangle$.  The groups $H_1$ and $H_2$ are the first of infinitely many pairs of nonisomorphic groups with the same order statistics (one can take $(\Z/p\Z)^3$ and the Heisenberg group $H_3(\Fp)$ for any prime $p$, for example).  Corollary~\ref{cor:regiso} implies that all such pairs $H_1$ and $H_2$ can be realized as part of a Gassmann triple $(G,H_1,H_2)$.
\end{remark}

The original motivation for studying rational equivalence stems from its relationship to zeta functions of number fields.
Recall that the \emph{Dedekind zeta function} of a number field $K$ is defined by
\[
\zeta_K(z)\coloneqq \prod_\p (1-N(\p)^{-z})^{-1},
\]
where $\p$ varies over primes of $K$ (nonzero prime ideals of its ring of integers $\O_K$) and $N(\p)\coloneqq [\O_K:\p]$ is the cardinality of the residue field at $\p$ (its absolute norm).
The Euler product for $\zeta_K(z)$ defines a holomorphic function on $\Re(z)>1$ that extends to a meromorphic function on $\C$ with a simple pole at $z=1$ whose residue is given by the \emph{analytic class number formula}:
\begin{equation}\label{eq:acnf}
\lim_{z\to 1^+}(z-1)\zeta_K(z) = \frac{2^{r}(2\pi)^sh_KR_K}{\#\mu(K)|D_K|^{1/2}}.
\end{equation}
Here $r$ and $s$ are the number of real and complex places of $K$ (its \emph{signature}), $h_K$ is the class number, $R_K$ is the regulator, $\mu(K)$ is the group of roots of unity in $K^\times$, and $D_K$ is the discriminant of $K$.

\begin{theorem}\label{thm:gassmann}
For number fields $K_1$ and $K_2$ the following are equivalent:
\begin{enumerate}[{\rm (i)}]
\setlength{\itemsep}{0pt}
\item $\zeta_{K_1}(s)=\zeta_{K_2}(s)$;
\item $K_1$ and $K_2$ have Galois closure $L$ with $\Gal(L/K_1), \Gal(L/K_2)\le \Gal(L/\Q)$ rationally equivalent;
\item There is a bijection between the primes of $K_1$ and $K_2$ that preserves residue fields.
\end{enumerate}
\end{theorem}
\begin{proof}
These equivalences all follow from \cite[Theorem 1]{Perlis77}.
\end{proof}

\begin{definition}
Number fields $K_1$ and $K_2$ that satisfy the equivalent conditions of Theorem~\ref{thm:gassmann} are said to be \emph{arithmetically equivalent}.
\end{definition}

If $K_1$ and $K_2$ are arithmetically equivalent number fields with common Galois closure $L$ and we put $G:=\Gal(L/\Q)$, $H_1:=\Gal(L/K_1)$, $H_2:=\Gal(L/K_2)$, then $(G,H_1,H_2)$ is a \emph{faithful} Gassmann triple, meaning that $\Q[H_1\backslash G]\simeq \Q[H_2\backslash G]$ is a faithful representation of $G$.
Equivalently, $H_1$ and $H_2$ have trivial normal core in $G$.
There is no loss of generality in restricting our attention to faithful Gassmann triples: if $H_1,H_2\le G$ are rationally equivalent then they necessarily have the same normal core $N$, the quotients $H_1/N,H_2/N\le G/N$ are rationally equivalent, and $H_1/N$ and $H_2/N$ are conjugate in $G/N$ if and only if $H_1$ and $H_2$ are conjugate in $G$.

Arithmetically equivalent number fields share many (but not all) arithmetic invariants.

\begin{theorem}\label{thm:arithequiv}
Arithmetically equivalent number fields have the same degree, discriminant, signature, and roots of unity.
\end{theorem}
\begin{proof}
See \cite[Theorem~1]{Perlis77}.
\end{proof}

The analytic class number formula \eqref{eq:acnf} implies that if $K_1$ and $K_2$ are arithmetically equivalent number fields then we must have
\[
h_{K_1}R_{K_1}=h_{K_2}R_{K_2},
\]
but it may happen that $h_{K_1}\ne h_{K_2}$ (in which case $R_{K_1}\ne R_{K_2}$), and even when $h_{K_1}=h_{K_2}$ the class groups need not be isomorphic.\footnote{The fields $\Q[x]/(x^7-3x^6+10x^5-21x^4-6x^3+58x^2-41x-6)$ and $\Q[x]/(x^7-x^6+x^5+5x^4+9x^3+5x^2-7x-4)$ with LMFDB \cite{LMFDB} labels \href{http://www.lmfdb.org/NumberField/7.3.1427382162361.1}{\texttt{7.3.1427382162361.1}} and \href{http://www.lmfdb.org/NumberField/7.3.1427382162361.2}{\texttt{7.3.1427382162361.2}} are an example; see \cite{BP} for analogous exceptions in the context of isospectral Riemannian manifolds.}
It follows from Theorem \ref{thm:arithequiv} that if $K_1$ and $K_2$ are arithmetically equivalent then a prime $p$ of $\Q$ ramifies in $K_1$ if and only if it ramifies in $K_2$.

\begin{remark}\label{rem:ramdiff}
The ramified rational primes in arithmetically equivalent number fields necessarily coincide, but they may have different factorization patterns.
This was shown by Perlis in \cite[page 351]{Perlis77} for the arithmetically equivalent number fields $K_1\coloneqq \Q(\sqrt[8]{97})$ and $K_2\coloneqq \Q(\sqrt[8]{1552})$ where we have
\[
2\O_{K_1} = \p_1\p_2\p_3^2\p_4^4\qquad\text{versus}\qquad 2\O_{K_2} = \q_1^2\q_2^2\q_3^2\q_4^2.
\]
In this example the products of the ramification indices differ, but the sums are the same. As shown by Mantilla-Soler \cite[Thm.\ 3.7]{Mantilla-Soler}, there are cases where the sums also differ.  Indeed, the number fields $K_1:=\Q[x]/(x^7 - 3x^6 + 4x^5 - 5x^4 + 3x^3 - x^2 - 2x + 1)$ and $K_2:=\Q[x]/(x^7 - x^5 - 2x^4 - 2x^3 + 2x^2 - x + 4)$ with LMFDB labels \href{https://www.lmfdb.org/NumberField/7.3.30558784.1}{\texttt{7.3.30558784.1}} and \href{https://www.lmfdb.org/NumberField/7.3.30558784.2}{\texttt{7.3.30558784.2}} are arithmetically equivalent with
\[
2\O_{K_1} = \p_1\p_2^4\qquad\text{versus}\qquad 2\O_{K_2} = \q_1\q_2^2.
\]
This example settled a question of Stuart and Perlis \cite[\S 4]{PerlisStuart}.
\end{remark}

For a number field $K$ with Galois closure $L$ that is the fixed field of $H\le G=\Gal(L/\Q)$, the decomposition of rational primes in $K$, can be computed using the $G$-set $[H\backslash G]$.  The lemma below can be used to explain the examples in Remark~\ref{rem:ramdiff} and to prove part (iii) of Theorem~\ref{thm:gassmann}.

\begin{lemma}
Let $L$ be a Galois extension of $\Q$ with Galois group $G$, let~$\p$ be a prime of $L$ above $p:=\p\cap\Q$ with decomposition group $D_\p$ and inertia group $I_\p$,
and let $K$ be the fixed field of $H\le G$.
\begin{enumerate}[{\rm (i)}]
\setlength{\itemsep}{0pt}
\item There is a bijection $[H\backslash G/D_\p]\to \{\text{primes of $K$ above $p$}\}$ defined by $H\sigma D_\p\mapsto \sigma(\p)\cap K$.
\item The prime $\sigma(\p)\cap K$ has ramification index $[H\sigma I_\p:H\sigma]$ and residue field degree $[H\sigma D_\p:H\sigma I_\p]$.
\end{enumerate}
\end{lemma}
\begin{proof}
This is well known; see \cite[\S 9]{Neukirch} and \cite{Wood}, for example.
\end{proof}

Theorem~\ref{thm:gassmann} implies that if $K_1$ and $K_2$ are arithmetically equivalent number fields then for every unramified rational prime $p$ there is a bijection between the primes of $K_1$ above $p$ and the primes of~$K_2$ above $p$ such that the completions of $K_1$ and $K_2$ at corresponding primes above $p$ are isomorphic extensions of $\Q_p$, since (up to isomorphism) there is a unique unramified extension of $\Q_p$ of each degree.  By Theorem~\ref{thm:arithequiv}, this also holds for the archimedean place $\infty$ of $\Q$, since the signatures of $K_1$ and $K_2$ coincide,
but as shown by the examples of Remark~\ref{rem:ramdiff}, this need not hold at ramified primes.

\begin{definition}
Two number fields $K_1$ and $K_2$ are said to be \emph{locally isomorphic} if there is a bijection between the places of $K_1$ and the places of $K_2$ such that the completions at corresponding places are isomorphic (both as topological rings and as $\Q_p$-algebras); equivalently, $K_1\otimes_\Q\Q_p\simeq K_2\otimes_\Q\Q_p$ for all $p\le \infty$.  If this holds for all but finitely many places then $K_1$ and $K_2$ are said to be \emph{locally isomorphic almost everywhere}.
\end{definition}

For a number field $K$ we use $\A_K$ to denote its ring of adeles, which we may regard both as a topological ring and as an $\A_\Q$-algebra.

\begin{theorem}\label{thm:localiso}
Let $K_1$ and $K_2$ be number fields.  The following hold:
\begin{enumerate}[{\rm (i)}]
\setlength{\itemsep}{0pt}
\item $K_1$ and $K_2$ are locally isomorphic almost everywhere if and only if they are arithmetically equivalent,
and if and only if almost every prime of $\Q$ has the same number of primes above it in $K_1$ and $K_2$;
\item $K_1$ and $K_2$ are locally isomorphic if and only if $\A_{K_1}\simeq \A_{K_2}$ (as topological rings and $\A_\Q$-algebras);
\item if $K_1$ and $K_2$ are locally isomorphic then there is a natural isomorphism of their Brauer groups that commutes with all restriction maps induced by common inclusions of number fields.
\end{enumerate}
\end{theorem}
\begin{proof}
The first equivalence in (i) follows from \cite[Theorem~1]{Perlis77} and the second was proved in~\cite{PerlisStuart}, the forward implication in (ii) is immediate and the reverse implication is due to Iwasawa \cite[Lemma 7]{Iwasawa} (also see \cite[Lemma 3]{Komatsu76}), and the implication in (iii) is proved in \cite{LMM}.
\end{proof}

\begin{remark}
The converse of part (iii) of Theorem~\ref{thm:localiso} is false.  Arithmetically equivalent number fields with naturally isomorphic Brauer groups need not be locally isomorphic, as shown in \cite{McReynolds}.
\end{remark}

Theorem \ref{thm:localiso} implies that locally isomorphic number fields are necessarily arithmetically equivalent, but the converse need not hold.  As observed in Remark~\ref{rem:ramdiff}, arithmetically equivalent number fields may have incompatible ramification indices, which precludes local isomorphism.

\begin{remark}\label{rem:localisoclassnum}
Locally isomorphic number fields need not have the same class number; the fields $\Q(\sqrt[8]{-33})$ and $\Q(\sqrt[8]{-528})$ with class numbers 256 and~128 are an example \cite[p.~214]{dSP}.
\end{remark}

The following proposition provides an effective way to test for local isomorphism.

\begin{proposition}\label{prop:decomp}
Let $L,K_1,K_2$ be number fields corresponding to a Gassmann triple $(G,H_1,H_2)$, and let $D_\p\subseteq G$ be the decomposition group of a place $\p$ of $L$ above a place $p$ of $\Q$.
Then $K_1\otimes_\Q \Q_p\simeq K_2\otimes_\Q \Q_p$ if and only if $[H_1\backslash G]$ and $[H_2\backslash G]$ are isomorphic as $D_\p$-sets.  These equivalent conditions necessarily hold for every unramified place $p$ of $\Q$.
\end{proposition}
\begin{proof}
Recall that for any field $F$ with separable closure $\Omega$ there is a functorial equivalence between the category of \'etale $F$-algebras $A$ and the category of finite $\Gal(\Omega/F)$-sets $S$; see \cite[Theorem~8.20]{Milne}. The $\Gal(\Omega/F)$-action on $S$ is continuous, hence factors through a finite quotient $Q$, and by a $Q$-set $S$ we mean the $\Gal(\Omega/F)$-set $S$ with the action of each $\sigma\in \Gal(\Omega/F)$ given by the action of its projection to $Q$.

For $i=1,2$, the $G$-set $[H_i\backslash G]$ corresponds to the \'etale $\Q$-algebra $K_i$.  If we view $D_\p$ as the Galois group of the \'etale $\Q_p$-algebra $L\otimes \Q_p$, the $D_\p$-set $[H_i\backslash G]$ corresponds to the \'etale $\Q_p$-algebra $K_i\otimes \Q_p$.

The last statement follows from (iv) of Proposition~\ref{prop:gassmann}, since if $\p$ is unramified then $D_\p$ is cyclic.
\end{proof}

Finally we recall the following result on arithmetical isomorphisms which can be found in \cite[IV]{Klingen}.
\begin{proposition}\label{prop:arithiso}
Let $G$ be a finite group with subgroups $H_1,H_2\le G$, let $R$ be an integral domain, let $A$ be an $R[G]$-module, and let $A_1\coloneqq A^{H_1}$ and $A_2\coloneqq A^{H_2}$ be the $R$-submodules of $A$ fixed by $H_1$ and~$H_2$, respectively.
Every $M\in \Hom_{R[G]}(R[H_1\backslash G],R[H_2\backslash G])$ with $\det M\in R^{\times}$ induces an $R[G]$-module isomorphism $\delta_M\colon A_1\to A_2$.
\end{proposition}
\begin{proof}
See \cite[Theorem IV.1.6a]{Klingen}.
\end{proof}
\section{Stronger forms of arithmetic equivalence}

Recall that a finite group $K$ is said to be \emph{cyclic modulo} $p$ (or $p$-\emph{hypo-elementary}) if the quotient of $K$ by the intersection of its $p$-Sylow subgroups (its $p$-core) is cyclic.  For the sake of brevity we shall simply call such a group \emph{$p$-cyclic}.  The class of $p$-cyclic groups includes all $p$-groups and all cyclic groups.

\begin{proposition}\label{prop:localequiv}
Let $G$ be a finite group and $p$ a prime.  For $H_1,H_2\le G$ the following are equivalent:
\begin{enumerate}[{\rm (i)}]
\setlength{\itemsep}{-2pt}
\item There is a $G$-class preserving bijection between the sets of $p$-cyclic subgroups of $H_1$ and $H_2$;
\item $\chi_{H_1}(K)=\chi_{H_2}(K)$ for every $p$-cyclic $K\le G$;
\item the $G$-sets $[H_1\backslash G]$ and $[H_2\backslash G]$ are isomorphic as $K$-sets for every $p$-cyclic $K\le G$;
\item $\Z_p[H_1\backslash G]\simeq \Z_p[H_2\backslash G]$;
\item $\F_p[H_1\backslash G]\simeq \F_p[H_2\backslash G]$;
\item $p\nmid d(H_1,H_2)$.
\end{enumerate}
Moreover, in (ii) and (iii) one can replace ``$K\le G$'' with ``$K\le H_1$ or $K\le H_2$''.
\end{proposition}
\begin{proof}
The equivalence of (i), (ii), (iii) is given by Proposition~\ref{prop:pred}.
The equivalence of (ii) and (iv) follows from \cite[Proposition~3.1]{Scott92} (attributed to Conlon \cite{Conlon}).
The equivalence of (iv) and (v) is given by \cite[Theorem~2.9(i)]{GW}.
The equivalence of (v) and (vi) is immediate, since $\F_p[H_1\backslash G]\simeq \F_p[H_2\backslash G]$ if and only if there exists $M\in \Hom_{\Z[G]}(\Z[H_1\backslash G],\Z[H_1\backslash G])$ whose reduction modulo $p$ is invertible, equivalently, $p\nmid \det M)$.
That the weakened forms of (ii) and (iii) suffice follows form Proposition~\ref{prop:pred}
\end{proof}

\begin{definition}
Let $H_1,H_2\le G$ be finite groups.  If $\Z_p[H_1\backslash G]\simeq \Z_p[H_2\backslash G]$ for every prime $p$ then $H_1$ and~$H_2$ are \emph{locally integrally equivalent}, and if $\Z[H_1\backslash G]\simeq \Z[H_2\backslash G]$ then they are \emph{integrally equivalent}.
\end{definition}

\begin{remark}
Two $\Z[G]$-modules that are isomorphic as $\Z_p[G]$-modules for every prime $p$ are said to \textit{lie in the same genus} \cite{GWe,Scott92}; subgroups $H_1,H_2\le G$ are locally integrally equivalent if and only of the permutation modules $\Z[H_1\backslash G]$ and $\Z[H_2\backslash G]$ lie in the same genus.
\end{remark}

Proposition~\ref{prop:localequiv} implies that subgroups $H_1,H_2\le G$ are locally integrally equivalent if and only if
\[
d(H_1,H_2)=\gcd \left\{\det M:M \in \Hom_{\Z[G]}(\Z[H_1\backslash G],\Z[H_2\backslash G])\right\}= 1,
\]
in which case there is a finite set of matrices $M\in \Hom_{\Z[G]}(\Z[H_1\backslash G],\Z[H_2\backslash G]$ whose determinants have trivial GCD.
Integral equivalence holds if and only if a singleton set with this property exists, that is, $\det M=\pm 1$ for some $M\in \Hom_{\Z[G]}(\Z[H_1\backslash G],\Z[H_2\backslash G])$.
Rational equivalence only requires $d(H_1,H_1)\ne 0$ and is obviously implied by local integral equivalence.

Essentially only one nontrivial example of integral equivalence is known, due to Scott \cite{Scott92}, in which $G\simeq \PSL_2(29)$ and $H_1$ and $H_2$ are nonconjugate subgroups of $G$ isomorphic to the alternating group~$A_5$ that are conjugate in $\PGL_2(29)$; one can use this example to construct others, but these all have a subgroup with a quotient isomorphic to $\PSL_2(29)$.
As noted by Scott and proved in Theorem~\ref{thm:solvequiv} below, for every prime $p\equiv \pm 29\bmod 120$ the group $\PSL_2(p)$ contains nonconjugate subgroups isomorphic to~$A_5$ that are locally integrally equivalent.  But with the exception of $p = 29$ it is not known whether these subgroups are also integrally equivalent.

\begin{proposition}\label{prop:localarithequiv}
Let $K_1$ and $K_2$ be number fields with common Galois closure $L$, and let $H_1:=\Gal(L/K_1)$, $H_2:=\Gal(L/K_2)$ be locally integrally equivalent subgroups of $G:=\Gal(L/\Q)$.
Then the following hold:
\begin{enumerate}[{\rm (i)}]
\setlength{\itemsep}{-2pt}
\item $K_1$ and $K_2$ are arithmetically equivalent;
\item the class groups of $K_1$ and $K_2$ are isomorphic;
\item the regulators of $K_1$ and $K_2$ are equal;
\item for every prime $p$ the products of the ramification indices of the primes of $K_1$ and $K_2$ above $p$ coincide.
\end{enumerate}
\end{proposition}
\begin{proof}
As noted above, local integral equivalence implies rational equivalence, so (i) follows from Proposition~\ref{prop:gassmann} and Theorem~\ref{thm:gassmann}.
Proposition~\ref{prop:localequiv} and \cite[Theorem~3]{Perlis78} together imply that the class groups of $K_1$ and $K_2$ have isomorphic $p$-Sylow subgroups for every prime $p$ and are therefore isomorphic (since they are abelian), so (ii) holds.
Properties (i) and (ii) together imply (iii), by Theorem~\ref{thm:arithequiv} and the analytic class number formula.
Local integral equivalence implies $d(H_1,H_2)=1$, which when combined with \cite[Theorem~IV.2.3]{Klingen} implies (iv).
\end{proof}

For number fields satisfying the hypothesis of Proposition~\ref{prop:localarithequiv}, all the quantities that appear in the analytic class number formula \eqref{eq:acnf} must coincide.
However, such fields need not be locally isomorphic, as shown by the example in \S\ref{sec:locnonlociso}, and locally isomorphic number fields may have different class numbers and regulators, as shown by the example in Remark~\ref{rem:localisoclassnum}.

We now introduce a strictly stronger notion of equivalence that implies both local integral equivalence and local isomorphism of  corresponding number fields.

\begin{definition}\label{def:solvequiv}
Subgroups $H_1$ and $H_2$ of a finite group $G$ are \emph{solvably equivalent} if they satisfy the following equivalent properties (as guaranteed by Proposition~\ref{prop:pred}):
\begin{enumerate}[{\rm (i)}]
\setlength{\itemsep}{-2pt}
\item There is a $G$-class preserving bijection between the sets of solvable subgroups of $H_1$ and $H_2$;
\item $\chi_{H_1}(K)=\chi_{H_2}(K)$ for every solvable $K\le G$;
\item the $G$-sets $[H_1\backslash G]$ and $[H_2\backslash G]$ are isomorphic as $K$-sets for every solvable $K\le G$.
\end{enumerate}
\end{definition}

Solvably equivalent subgroups are always locally integrally equivalent, since $p$-cyclic groups are solvable, but as demonstrated by the example in \S\ref{sec:d32}, locally integrally equivalent subgroups need not be solvably equivalent.
As shown by the example in \S\ref{sec:d96}, solvably equivalent subgroups need not be integrally equivalent, but it is not clear whether the converse holds; the integrally equivalent subgroups of $\PSL_2(29)$ in Scott's example are solvably equivalent, but as noted in the introduction, it is not clear whether this is always true, nor is it clear that integral equivalence guarantees local isomorphism of corresponding number fields (this is not true of local integral equivalence, and if it were true for integral equivalence then property~(2) in Theorem 3 in \cite{Prasad} could have been included in Theorem 2 in \cite{Prasad}).

\begin{question}\label{quest:intsolv}
Is there a Gassmann triple $(G, H_1, H_2)$ in which $H_1$ and $H_2$ are integrally equivalent but not solvably equivalent?
More precisely, is there a group $G$ containing subgroups $H_1,H_2$ and a solvable subgroup $K$ such that $\Z[H_1\backslash G]$ and $\Z[H_2\backslash G]$ are isomorphic as $\Z[G]$-modules but not as $K$-sets?
\end{question}

\begin{proposition}\label{prop:solvablearithequiv}
Let $K_1$ and $K_2$ be number fields with the same Galois closure $L$, and put $H_1:=\Gal(L/K_1)$ and $H_2:=\Gal(L/K_2)$.
If $H_1$ and $H_2$ are solvably equivalent subgroups of $G:=\Gal(L/\Q)$ then
\begin{enumerate}[{\rm (i)}]
\setlength{\itemsep}{-2pt}
\item $K_1$ and $K_2$ are arithmetically equivalent;
\item $K_1$ and $K_2$ have isomorphic class groups and equal regulators;
\item $K_1$ and $K_2$ are locally isomorphic, and in particular there is a bijection between the primes of $K_1$ and~$K_2$ that preserves both inertia degrees and ramification indices;
\item the adele rings $\A_{K_1}$ and $\A_{K_2}$ are isomorphic (as topological groups and $\A_\Q$-algebras);
\end{enumerate}
\end{proposition}
\begin{proof}
Solvable equivalence implies local integral equivalence, so (i) and (ii) both follow from Proposition~\ref{prop:localarithequiv}.
For each prime $\p$ of $L$ the decomposition subgroup $D_\p\subseteq \Gal(L/\Q)$ is solvable, so we have an isomorphism of $D_\p$-sets $[H_1\backslash G]\simeq [H_2\backslash G]$, which implies (iii), by Proposition~\ref{prop:decomp}, and (iv) is then implied by Theorem~\ref{thm:localiso}.  
\end{proof}

\begin{remark}
In Proposition~\ref{prop:solvablearithequiv}, the hypothesis that $H_1$ and $H_2$ are solvably equivalent is stronger than necessary.  It could be replaced, for example, by the condition that $\chi_{H_1}(K)=\chi_{H_2}(K)$ for every $K\le G$ with normal subgroups $W\le I$ such that $W$ is a $p$-group, $I/W$ is cyclic of order prime to $p$, and $K/I$ is cyclic.
Even this is stronger than necessary, since, for example, it is satisfied by both $C_2^4$ and $\SL_2(3)$, neither of which occurs as the Galois group of an extension of $\Q_p$ for any prime $p$ (the former contains too many normal subgroups of index 2 and the latter was ruled out by Weil in \cite[\S 15]{Weil}).
\end{remark}

The following theorem gives an infinite family of groups each of which contain a pair of nonconjugate solvably equivalent subgroups.

\begin{theorem}\label{thm:solvequiv}
Let $p\equiv \pm 29 \bmod 120$ be prime.
The group $\SL_2(\F_p)$ contains a pair of nonconjugate solvably equivalent subgroups $H_1, H_2$ whose projective images are nonconjugate solvably equivalent subgroups of $\PSL_2(\F_p)$ isomorphic to the alternating group $A_5$.
\end{theorem}
\begin{proof}
It follows from \cite[Lemma 3.21.3c]{Sutherland} that for $p\equiv \pm 1\bmod 5$, up to conjugacy in $\GL_2(\Fp)$ there is a unique subgroup $H_1$ of $\SL_2(\Fp)$ with projective image isomorphic to $A_5$; it is isomorphic to $\SL_2(\F_5)$.  The outer automorphism of $\SL_2(\F_p)$ corresponds to conjugation by an element with nonsquare determinant; let $\sigma\coloneqq \smat{r}{0}{0}{1}$ be such an element, with $r\in \Fp^\times - \Fp^{\times 2}$.
Conjugation by $\sigma$ fixes all but four of the conjugacy classes in $\SL_2(\Fp)$: it interchanges the conjugacy classes of $\smat{1}{1}{0}{1}$ and $\smat{1}{r}{0}{1}$, and also those of $\smat{-1}{-1}{0}{-1}$ and $\smat{-1}{-r}{0}{-1}$ (these are the conjugacy classes of elements of order divisible by $p$).

Let $H_2\coloneqq \sigma H_1\sigma^{-1}$; the groups $H_1$ and $H_2$ are not conjugate in $\SL_2(\Fp)$, by \cite[Theorem~4.1]{Flicker}.  These groups do not contain any elements of order divisible by $p$, since $p\ge 29$ and $\#\SL_2(\F_5)=2^2\cdot3\cdot5$.   Conjugation by $\sigma$ thus defines an $\SL_2(\Fp)$-conjugacy preserving bijection between $H_1$ and $H_2$, implying that $H_1$ and $H_2$ are rationally equivalent subgroups of $\SL_2(\Fp)$.

To show that $H_1$ and $H_2$ are solvably equivalent, it suffices to show that $\sigma$ defines an $\SL_2(\Fp)$-conjugacy class preserving bijection of solvable subgroups of $H_1$ and $H_2$, and having proved rational equivalence we only need to consider the noncyclic solvable subgroups of $H_1$ and $H_2$. Up to isomorphism, there are four possibilities for the image of such a subgroup in in $\PSL_2(\Fp)$: $D_2$, $D_3$, $D_5$, and $A_4$, where $D_2\coloneqq C_2\times C_2$ is the Klein group.
It follows from Proposition~\ref{lem:sl2subgroups} below that there is exactly one $\SL_2(\Fp)$-conjugacy class of subgroups isomorphic to $D_2$, $D_3$, $D_5$, $A_4$ when $p\equiv \pm 3\bmod 8$, $p\equiv \pm 5 \bmod 12$, $p\equiv \pm 9\bmod 20$, and $p\equiv\pm 3\bmod 8$, respectively.
These constraints are simultaneously met precisely when $p\equiv \pm 29\bmod 120$, and in this situation it is clear that $\sigma$ must define an $\SL_2(\Fp)$-conjugacy class preserving bijection of solvable subgroups of $H_1$ and $H_2$, since it preserves isomorphism classes.

Finally, note that the conjugacy class preserving bijection between solvable subgroups of $H_1$ and $H_2$ descends to $\PSL_2(\Fp)$, while $H_1$ and $H_2$ both contain $-1$ and remain nonconjugate in $\PSL_2(\Fp)$.
\end{proof}

\begin{remark}
Theorem~\ref{thm:solvequiv} accounts for all nontrivial pairs of solvably equivalent subgroups of $\SL_2(\F_p)$, in fact all nontrivial pairs of locally integrally equivalent subgroups of $\SL_2(\F_p)$, as noted by Scott~\cite{Scott92}.  Up to a central extension the same applies to subgroups of $\GL_2(\Fp)$, since every nonsolvable subgroup of $\GL_2(\Fp)$ that does not contain $\SL_2(\Fp)$ has projective image $A_5$ \cite[\S 2]{Serre}.
\end{remark}

\begin{remark}
As proved by Zywina \cite{Zywina}, the group $\PSL_2(\Fp)$ can be realized as the Galois group of a number field for every prime $p$.
This implies that there are infinitely many distinct examples of pairs of nonisomorphic solvably equivalent number fields whose Galois groups do not admit a common quotient.
\end{remark}

\begin{remark}
As shown in \S\ref{sec:d96}, subgroups of $\SL_2(\Fp)$ are not the only source of nontrivial solvably equivalent pairs of subgroups, and one can do better than the minimal degree 203 admitted by Theorem~\ref{thm:solvequiv}: degree 96 is possible.
\end{remark}

Recall that each subgroup of $\GL_2(\Fp)$ of order prime to $p$ can be classified according to the isomorphism class of its image in $\PGL_2(\Fp)$, which must be cyclic, dihedral, or one of $A_4$, $S_4$, $A_5$; see \cite[\S 2]{Serre}, for example. Note that we consider $D_2:= C_2\times C_2$ to be a dihedral group.  The proposition below characterizes the isomorphism classes of order prime to $p$ that arise in $\SL_2(\Fp)$, up to conjugacy in $\SL_2(\Fp)$; see \cite[\S 3]{Sutherland} for an analogous classification for conjugacy classes of subgroups of $\GL_2(\Fp)$ (including those of order divisible by $p$), which we will use in the proof of the proposition.

We use the notation $2D_n$ to denote the binary dihedral group of order $4n$, these arise as subgroups of $\SL_2(\Fp)$ containing $-1$ with projective image $D_n$, and similar define $2A_4$, $2S_4$, $2A_5$. We say that a conjugacy class of subgroups of $\SL_2(\Fp)$ is $C_n$ (resp. $2D_n$, $2A_4$, $2S_4$, $2A_5$) if it is the conjugacy class of a subgroup isomorphic to $C_n$ (resp. $2D_n$, $2A_4$, $2S_4$, $2A_5$).

\begin{proposition}\label{lem:sl2subgroups}
Let $p>3$ be prime, and let $S$ be the set of integers that divide either $p-1$ or $p+1$.
Up to conjugacy in $\SL_2(\Fp)$ the subgroups of $\SL_2(\Fp)$ of order prime to $p$ are as follows:
\begin{itemize}
\setlength{\itemsep}{-2pt}
\item For each integer $n\ge 1$ with $p\equiv \pm 1 \bmod n$, a single conjugacy class $C_n$.
\item For each integer $2n>2$ with $p\equiv \pm 1 \bmod 4n$, two conjugacy classes $2D_n$.
\item For each integer $2n>2$ with $p\equiv \pm 1 \bmod 2n$ and $p\not\equiv\pm 1\bmod 4n$, a single conjugacy class $2D_n$.
\item Two conjugacy classes $2A_4$ if $p\equiv \pm 1\bmod 8$ and one otherwise.
\item Two conjugacy classes $2S_4$ if $p\equiv \pm 1\bmod 8$ and none otherwise.
\item Two conjugacy classes $2A_5\simeq \SL(2,5)$ if $p\equiv \pm 1\bmod 5$ and none otherwise.
\end{itemize}
\end{proposition}
\begin{proof}
Every cyclic subgroup of $\SL_2(\Fp)$ order prime to $p$ must be conjugate in $\GL_2(\Fp)$ to a subgroup of one of the two Cartan subgroups $C$: the \emph{split Cartan} isomorphic to $\Fp^\times\times \Fp^\times$, or the \emph{nonsplit Cartan} isomorphic to $\F_{p^2}^\times$.  The intersection of $C$ with $\SL_2(\Fp)$ is cyclic of order $p-1$ or $p+1$, and the intersection of these groups is the cyclic group $\{\pm 1\}$ of order $2=\gcd(p-1,p+1)$.  It follows that up to $\GL_2(\Fp)$-conjugacy there is a unique cyclic subgroup $C_n$ of $\SL_2(\Fp)$ of order $n$ for each $n$ dividing $p-1$ or $p+1$, and \cite[Theorem~4.1]{Flicker} implies that it is also unique up to $\SL_2(\Fp)$-conjugacy.

For a Cartan subgroup $C$ of $\GL_2(\Fp)$, let $C^+$ denote its normalizer. It follows from \cite[Lemma~3.13]{Sutherland} that for each subgroup $H$ of $C\cap \SL_2(\Fp)$ there is at most one subgroup $G$ of $C^+\cap \SL_2(\Fp)$ with dihedral image in $\PSL(2,p)$, and that subgroup must contain $-1$.  It follows from \cite[Lemmas 3.16 and 3.18]{Sutherland} that there is exactly one $G$ for each $H\ne\{\pm 1\}$ that contains $-1$, up to conjugacy in $\GL_2(\Fp)$.
It follows that for each integer $2n>2$ dividing $p-1$ or $p+1$ that up to $\GL_2(\Fp)$-conjugacy there is a unique conjugacy class $2D_n$ of $\SL_2(\Fp)$, and it follows from \cite[Theorem~4.1]{Flicker} and Remark~\ref{rem:flickererratum} below that this $\GL_2(\Fp)$-conjugacy class splits into two $\GL_2(\Fp)$-conjugacy classes if and only if $p\equiv \pm 1\bmod 4n$.

The statements for $2A_4$, $2S_4$, $2A_5$ are immediate from \cite[Lemma~3.21]{Sutherland} and \cite[Theorem~4.1]{Flicker}.
\end{proof}

\begin{remark}\label{rem:flickererratum}
There is a minor error in the statement \cite[Theorem~4.1]{Flicker} regarding the group~$2D_2$, which is denoted $BD_{4\cdot 2}$ in \cite{Flicker}.  There are two conjugacy classes $2D_2$ in $\SL_2(\Fp)$ when $\sqrt{2}\in \Fp$, equivalently, when $p\equiv \pm 1\bmod 8$, but only one otherwise; this follows from the fact that the normalizer of $BD_{4\cdot 2}$ in $\SL_2(\overline{\F}_p)$ is $2S_4$ (not $2A_4$ as claimed in \cite{Flicker}), which is present in $\SL_2(\Fp)$ only when $\sqrt{2}\in \Fp$. The author is grateful to Yuval Flicker for clarifying this point.
\end{remark}

\section{Computational results}
In this section we present examples that realize the claims made in the introduction, including that local integral equivalence does not imply group isomorphism (\S\ref{sec:locnoniso}),  local isomorphism of number fields (\S\ref{sec:locnonlociso}), or integral equivalence (\S\ref{sec:d32}), and that solvable equivalence does not imply integral equivalence~(\S\ref{sec:d96}).  We also give a degree 32 example of locally integrally equivalent number fields in~\S\ref{sec:d32} (best possible), and a degree 96 example of solvably equivalent number fields in~\S\ref{sec:d96} (best known).

\subsection{Locally integrally equivalent subgroups need not be isomorphic}\label{sec:locnoniso}

In \cite[Question~1]{Prasad}, Prasad asks if integrally equivalent subgroups are necessarily isomorphic.  This is true in Scott's example with two subgroups of $\PSL_2(\F_{29})$ isomorphic to the alternating group $A_5$.  The following example shows that locally integrally equivalent subgroups need not be isomorphic.  Let $G$ by the symmetric group $S_{21}$ and consider the subgroups
\begin{align*}
H_1 \coloneqq \bigl\langle\, &(4\ 5)(6\ 15\ 7\ 14)(8\ 17\ 9, 16)(10\ 19\ 11\ 18)(12\ 21\ 13\ 20),\\
               &(1\ 2)(3\ 5)(6\ 20\ 8\ 18)(7\ 21\ 9\ 19)(10\ 14\ 12\ 16)(11\ 15\ 13\ 17)\,\bigr\rangle,\\
H_2 \coloneqq \bigl\langle\, &(4\ 5)(6\ 16\ 8\ 14)(7\ 17\ 9\ 15)(10\ 20\ 12\ 18)(11\ 21\ 13\ 19),\\
               &(1\ 2)(3\ 5)(6\ 20\ 8\ 18)(7\ 21\ 9\ 19)(10\ 17\ 12\ 15)(11\ 16\ 13\ 14)\,\bigr\rangle,
\end{align*}
with GAP identifiers $\langle 48,12\rangle$ and $\langle 48,13\rangle$, respectively. Each contains 41 subgroups that are $p$-cyclic for some prime~$p$.  These fall into 15 distinct $G$-conjugacy classes and 11 distinct isomorphism classes, which makes it easy to find a $G$-conjugacy class preserving bijection between them (if one takes into account the isomorphism class and the number of subgroups in each conjugacy classes, there are only 2 choices to consider).  The subgroups $H_1,H_2\le G$ are thus locally integrally equivalent, but not isomorphic.  This negatively answers Question 2.11 posed by Guralnick and Weiss in \cite{GWe}.

This example is realized by infinitely many number fields: over $\Q$ the Galois group of a generic polynomial of degre~21 is $G=S_{21}$ and the fixed fields of $H_1$ and $H_2$ are locally integrally equivalent number fields of degree $21!/48$.  It is one of many that were found by applying Corollary~\ref{cor:regiso} to the clas~$\cP$ of groups that are $p$-cyclic for some prime $p$: computing $\cP$-statistics for the isomorphism classes of groups of order up to 255 already finds 107 pairs of isomorphism classes with the same $\cP$-statistics, including four isomorphism classes of groups of order 192 with the same $\cP$-statistics.  One can often find permutation representations of degree less than $|H_1|=|H_2|$ that also work, as happens above.

\begin{question}\label{ques:solvequiviso}
Are solvably equivalent subgroups of a finite group $G$ necessarily isomorphic?
\end{question}

Question~\ref{ques:solvequiviso} is equivalent to asking whether the isomorphism class of a nonsolvable group determine its $\cP$-statistics, where $\cP$ is the class of solvable groups (by Corollary~\ref{cor:regiso}).  For the 1022 isomorphism classes of nonsolvable groups of order less than 2000, these $\cP$-statistics are all distinct, so any pair of nonisomorphic solvably equivalent subgroups must have order greater than 2000.

\subsection{Local integral equivalence does not imply local isomorphism of number fields}\label{sec:locnonlociso}

Let $G$ be the group $A_4\times S_5$ with GAP identifier $\langle 1440,5846\rangle$.  There is a unique pair of nonconjugate locally integrally equivalent subgroups $H_1,H_2\le G$, both of which are isomorphic to the dihedral group $D_6$ of order~12.
The groups $G$, $H_1$, $H_2$ can be explicitly represented as subgroups of $S_9$ via
\begin{align*}
G   &\coloneqq \bigl\langle (1\ 2\ 3)(5\ 6\ 7\ 8\ 9),\,(1\ 2)(3\ 4)(5\ 6)\bigr\rangle,\\
H_1 &\coloneqq \bigl\langle (1\ 2)(3\ 4)(5\ 6\ 7)(8\ 9),\, (1\ 3)(2\ 4)(5\ 6)\bigr\rangle,\\
H_2 &\coloneqq \bigl\langle (1\ 2)(3\ 4)(5\ 6\ 7)(8\ 9),\, (1\ 4)(2\ 3)(5\ 6)\bigr\rangle,
\end{align*}
and $H_1\cap H_2$ is cyclic of order 6.  The four maximal subgroups of $H_1$, isomorphic to $C_2^2, S_3, S_3, C_6$, correspond to distinct conjugacy classes of subgroups of $G$, and these are precisely the $G$-conjugacy classes of the four maximal subgroups of $H_2$.  There is thus a $G$-conjugacy preserving bijection between the proper subgroups of $H_1$ and $H_2$ (all of which are $p$-cyclic for some prime $p$), and the group $D_6\simeq H_1,H_2$ is not $p$-cyclic for any prime $p$.  It follows that $H_1$ and $H_2$ are locally integrally equivalent subgroups of $G$.

The subgroups $H_1$ and $H_2$ are not $G$-conjugate, even though they are $S_9$-conjugate, as can be verified by comparing their permutation characters: $\chi_{H_1}(H_1)=4$ differs from $\chi_{H_2}(H_1)=0$, and $\chi_{H_1}(H_2)=0$ differs from $\chi_{H_2}(H_2)=4$.  The group $D_6$ arises as a Galois group of extensions of $\Q_p$ for $p\not\equiv 1\bmod 6$, and it follows from Proposition~\ref{prop:decomp} that if $H_1$ is the decomposition group of a prime above $p$ in a Galois extension $L/\Q$ with Galois group $G$, then the fixed fields $K_1\coloneqq L^{H_1}$ and $K_2\coloneqq L^{H_2}$ are locally integrally equivalent fields that cannot be locally isomorphic because four primes of $K_1$ above~$2$ must have residue field degree 1 and ramification index 1 (corresponding to the four cosets in $[H_1\backslash G]$ fixed by $H_1$), but no primes of $K_2$ above $2$ can have residue field degree 1 and ramification index 1.

To realize such an example it suffices to find a pair of linearly disjoint $A_4$ and $S_5$ fields such that that there is a prime of the compositum with decomposition group conjugate to $H_1$ or~$H_2$.  A search of $A_4$ and $S_5$ fields in the $L$-functions and modular forms database (LMFDB) unramified away from 2,3,5,7 finds a suitable pair: we may take the Galois closures for the fields $F_1\coloneqq \Q[x]/(x^4 - 6x^2 - 8x + 60)$ and $F_2\coloneqq \Q[x]/(x^5+5x^3+10x-2)$ with LMFDB labels \href{https://www.lmfdb.org/NumberField/4.0.254016.2}{\texttt{4.0.254016.2}} and \href{https://www.lmfdb.org/NumberField/5.1.500000.1}{5.1.500000.1}, respectively.
The compositum of their Galois closures is a degree 1440 number field $L$ with Galois group $G$.
The 120 primes of $L$ above $2$ all have residue degree 2, ramification index 6, decomposition group conjugate to~$H_1$, and inertia group conjugate to $H_1\cap H_2$; the local algebra $L\otimes_\Q \Q_2$ is isomorphic to $k^{120}$, where $k$ is the unique $D_6$-extension of $\Q_2$ of degree 12 containing $\Q_2(\sqrt{2})$, with LMFDB label \href{https://www.lmfdb.org/LocalNumberField/2.12.22.60}{\texttt{2.12.22.60}}.

Using the \texttt{GaloisSubgroup} function in Magma \cite{Magma} one can compute defining polynomials of degree 120 for the number fields $K_1\coloneqq L^{H_1}$ and $K_2\coloneqq L^{H_2}$, and using the $p$-adic valuation \texttt{extensions} method in Sage \cite{Sage} one can determine the residue field degrees and ramification indices of the primes above~2 in $K_2$ and $K_2$ by computing all extensions of the $2$-adic valuation of $\Q$ to $K_1$ and $K_2$.  We have
\begin{align*}
2\O_{K_1} &= \p_1\p_2\p_3\p_4\p_5^6\p_6^6\p_7^6\p_8^6\p_9^6\p_{10}^6\p_{11}^6\p_{12}^6\p_{13}^2\p_{14}^2\p_{15}^3\p_{16}^3\p_{17}^6\p_{18}^6\p_{19}^6\p_{20}^6,\\
2\O_{K_2} &= \q_1^2\q_2^2\q_3^2\q_4^2\q_5^3\q_6^3\q_7^3\q_8^3\q_9^6\q_{10}^6\q_{11}^6\q_{12}^6\q_{13}\q_{14}\q_{15}^6\q_{16}^6\q_{17}^6\q_{18}^6\q_{19}^6\q_{20}^6,
\end{align*}
where the primes $\p_i$ of $K_1$ and $\q_i$ of $K_2$ have residue degree 1 for $i\le 12$ and residue degree~$2$ for $i>12$.

\begin{remark}
This example can be viewed as a refinement of the example of Mantilla-Soler \cite{Mantilla-Soler} noted in Remark~\ref{rem:ramdiff}: the sums 82 and 86 of the ramification indices differ. But in the Mantilla-Soler example the products of the ramification indices also differ, which is possible because the subgroups are rationally equivalent but not locally integrally equivalent. Proposition~\ref{prop:localarithequiv} shows that this is not possible when the subgroups are locally integrally equivalent.  To our knowledge, this is the first example of a pair of arithmetically equivalent number fields and a prime $p$ for which the sums of the ramification indices of the primes above $p$ differ but the products do not.
\end{remark}

Finally, we note that the groups $H_1$ and $H_2$ are isomorphic to $D_6$, hence solvable, but the values of the permutation characters $\chi_{H_1}$ and $\chi_{H_2}$ differ on these groups, as noted above, so they are not solvably equivalent, which shows that solvable equivalence is a strictly stronger condition (as one would expect).

\subsection{A minimal degree example of local integral equivalence}\label{sec:d32}
An exhaustive search of isomorphisms classes of groups of order less than 1024 in the small groups database \cite{BEO} finds 74 groups $G$ that contain nonconjugate $H_1, H_2\le G$ that are locally integrally equivalent and have trivial normal core in~$G$ (meaning that $(G,H_1,H_2)$ is a faithful Gassmann triple).  The order of $G$ is necessarily not a prime power, since $p$-groups can be locally integrally equivalent only if they are conjugate, so only 1,206,112 of the 11,759,892 groups of order less than 1024 need to be checked.  Of these, two have order 384, seventeen have order 576, fifty have order 768, and five have order 864, with the index of $H_1,H_2$ in $G$ taking values in $\{32,48,64,72\}$.

The two groups $G$ of order 384 have GAP identifiers $\langle 384,18046\rangle$ and $\langle 384, 18050\rangle$, and are isomorphic to transitive permutation groups of degree 32 with LMFDB labels \href{https://www.lmfdb.org/GaloisGroup/32T9403}{\texttt{32T9403}} and \href{https://www.lmfdb.org/GaloisGroup/32T9408}{\texttt{32T9408}}, following the labeling convention in \cite{CH}.  Both are (nonsplit) 2-extensions of $D_4\times S_4$, making it feasible to explicitly construct examples of nonconjugate number fields $K_1$ and $K_2$ of degree 32 with common Galois closure $L$ with $G=\Gal(L/\Q)$, and $H_1=\Gal(L/K_1)$ and $H_2=\Gal(L/K_2)$ locally integrally equivalent, by taking a quadratic extension of the compositum of the Galois closure of two suitably chosen $D_4$ and~$S_4$ quartic number fields.  Below we describe one such example in detail.

The Magma computer algebra system \cite{Magma} includes a database of transitive permutation groups of degree up to 48 whose construction is described in \cite{CH,HRT,Hulpke}.  An exhaustive analysis of the 40,238 transitive groups of degree less than 32 finds none that contain a pair of locally integrally equivalent subgroups of index equal to the degree.  The following example thus achieves the minimal possible degree~$32$; for comparison, the minimal degree of arithmetically equivalent number fields is $7$; see \cite{BdS}.

We begin with the $D_4$ field $\Q[x]/(x^4-6x^2-9)$ and the $S_4$ field $\Q[x]/(x^4 - 2x^3 - 6x + 3)$, with LMFDB labels \href{https://www.lmfdb.org/NumberField/4.2.9216.1}{\texttt{4.2.9216.1}} and \href{https://www.lmfdb.org/NumberField/4.2.3888.1}{\texttt{4.2.3888.1}}, which are linearly disjoint over $\Q$.
The compositum of their Galois closures coincides with the splitting field of the polynomial
\[
x^{16} + 12x^{14} + 72x^{12} + 120x^{10} - 234x^8 + 108x^6 + 396x^4 - 432x^2 + 81,
\]
which has Galois group $D_4\times S_4$.  The number fields $K_1\coloneqq \Q[x]/(f_1(x))$, $K_2\coloneqq \Q[x]/(f_2(x))$ defined by
\begin{align*}
f_1&\coloneqq x^{32} + 12x^{28} + 72x^{24} + 120x^{20} - 234x^{16} + 108x^{12} + 396x^8 - 432x^4 + 81,\\
f_2&\coloneqq x^{32} - 12x^{28} + 72x^{24} - 120x^{20} - 234x^{16} - 108x^{12} + 396x^8 + 432x^4 + 81,
\end{align*}
have the same Galois closure $L$ of degree 384.
The group $G\coloneqq \Gal(L/\Q)$ is the transitive permutation group \href{https://www.lmfdb.org/GaloisGroup/32T9403}{\texttt{32T9403}}, generated by
\begin{footnotesize}
\begin{align*}
\sigma_0 &\coloneqq (3, 4, 5, 6, 7, 8)(9, 10, 11, 12, 13, 14)(15, 16, 17, 18, 19, 20)(21, 22, 23)(24, 25, 26)(27, 28)(29, 30)(31, 32),\\
\sigma_1 &\coloneqq (3, 5)(6, 8)(9, 10)(11, 14)(12, 13)(15, 17)(18, 20)(21, 24)(22, 26)(23, 25)(27, 31)(28, 32),\\
\sigma_2 &\coloneqq (1, 2)(3, 17)(4, 16)(5, 15)(6, 20)(7, 19)(8, 18)(9, 13)(10, 12)(22, 23)(25, 26)(29, 30),\\
\sigma_3 &\coloneqq (1, 3, 2, 15)(4, 9, 16, 12)(5, 24, 17, 21)(6, 30, 18, 29)(7, 22, 19, 25)(8, 14, 20, 11)(10, 32, 13, 28)(23, 27, 26, 31).
\end{align*}
\end{footnotesize}
\vspace{-12pt}

The group $G$ contains exactly two conjugacy classes of subgroups of index 32 with trivial normal core, represented by $H_1\coloneqq \langle \sigma_1,\sigma_2\rangle$ and $H_2\coloneqq \langle \sigma_0,\sigma_2\rangle$, both isomorphic to $D_6$.
If we view $G$ as acting on the roots of $f_1(x)$, then under a suitable ordering of roots we have $H_1=\Gal(L/K_1)$ and $H_2=\Gal(L/K_2)$.
The subgroups $H_1$ and $H_2$ are locally integrally equivalent but not integrally equivalent.
Indeed, for a suitable choice of bases for $[H_1\backslash G]$ and $[H_2\backslash G]$, every $M\in \Hom_{\Z[G]}(\Z[H_1\backslash G],\Z[H_2\backslash G])$ has the form

\begin{footnotesize}
\arraycolsep=2pt
\def\arraystretch{0.75}
\[
M:=
\begin{bmatrix}
x_8 & x_8 & x_5 & x_8 & x_8 & x_7 & x_8 & x_5 & x_8 & x_2 & x_1 & x_8 & x_6 & x_8 & x_8 & x_7 & x_7 & x_8 & x_7 & x_4 & x_3 & x_3 & x_1 & x_7 & x_8 & x_6 & x_2 & x_4 & x_5 & x_6 & x_7 & x_8\\
x_8 & x_8 & x_7 & x_8 & x_8 & x_5 & x_8 & x_7 & x_8 & x_3 & x_4 & x_8 & x_7 & x_8 & x_8 & x_5 & x_6 & x_8 & x_5 & x_4 & x_2 & x_2 & x_1 & x_6 & x_8 & x_7 & x_3 & x_1 & x_7 & x_7 & x_6 & x_8\\
x_7 & x_6 & x_8 & x_3 & x_7 & x_8 & x_5 & x_1 & x_2 & x_8 & x_7 & x_5 & x_8 & x_3 & x_7 & x_8 & x_8 & x_2 & x_1 & x_5 & x_8 & x_8 & x_6 & x_8 & x_6 & x_4 & x_8 & x_7 & x_8 & x_8 & x_4 & x_7\\
x_8 & x_1 & x_6 & x_8 & x_8 & x_2 & x_8 & x_5 & x_8 & x_7 & x_8 & x_4 & x_5 & x_8 & x_1 & x_7 & x_7 & x_8 & x_7 & x_8 & x_5 & x_6 & x_8 & x_2 & x_8 & x_6 & x_7 & x_8 & x_3 & x_3 & x_7 & x_4\\
x_8 & x_8 & x_6 & x_8 & x_8 & x_7 & x_8 & x_6 & x_8 & x_2 & x_4 & x_8 & x_5 & x_8 & x_8 & x_7 & x_7 & x_8 & x_7 & x_1 & x_3 & x_3 & x_4 & x_7 & x_8 & x_5 & x_2 & x_1 & x_6 & x_5 & x_7 & x_8\\
x_5 & x_7 & x_8 & x_2 & x_6 & x_8 & x_7 & x_4 & x_3 & x_8 & x_5 & x_7 & x_8 & x_2 & x_5 & x_8 & x_8 & x_3 & x_1 & x_7 & x_8 & x_8 & x_7 & x_8 & x_7 & x_1 & x_8 & x_6 & x_8 & x_8 & x_4 & x_6\\
x_8 & x_4 & x_7 & x_8 & x_8 & x_3 & x_8 & x_7 & x_8 & x_5 & x_8 & x_1 & x_7 & x_8 & x_1 & x_5 & x_6 & x_8 & x_6 & x_8 & x_7 & x_7 & x_8 & x_3 & x_8 & x_7 & x_6 & x_8 & x_2 & x_2 & x_5 & x_4\\
x_8 & x_8 & x_7 & x_8 & x_8 & x_6 & x_8 & x_7 & x_8 & x_3 & x_1 & x_8 & x_7 & x_8 & x_8 & x_6 & x_5 & x_8 & x_6 & x_1 & x_2 & x_2 & x_4 & x_5 & x_8 & x_7 & x_3 & x_4 & x_7 & x_7 & x_5 & x_8\\
x_5 & x_7 & x_8 & x_7 & x_6 & x_8 & x_7 & x_8 & x_6 & x_1 & x_3 & x_7 & x_8 & x_7 & x_6 & x_8 & x_8 & x_5 & x_8 & x_2 & x_1 & x_4 & x_2 & x_8 & x_7 & x_8 & x_4 & x_3 & x_8 & x_8 & x_8 & x_5\\
x_4 & x_8 & x_2 & x_8 & x_1 & x_6 & x_4 & x_7 & x_8 & x_5 & x_8 & x_8 & x_2 & x_8 & x_8 & x_3 & x_3 & x_8 & x_5 & x_8 & x_7 & x_7 & x_8 & x_5 & x_1 & x_7 & x_6 & x_8 & x_7 & x_7 & x_6 & x_8\\
x_7 & x_3 & x_8 & x_5 & x_7 & x_1 & x_6 & x_8 & x_7 & x_8 & x_7 & x_3 & x_8 & x_6 & x_2 & x_8 & x_8 & x_7 & x_8 & x_5 & x_8 & x_8 & x_6 & x_4 & x_5 & x_8 & x_8 & x_7 & x_1 & x_4 & x_8 & x_2\\
x_4 & x_8 & x_3 & x_8 & x_1 & x_7 & x_1 & x_6 & x_8 & x_7 & x_8 & x_8 & x_3 & x_8 & x_8 & x_2 & x_2 & x_8 & x_7 & x_8 & x_5 & x_6 & x_8 & x_7 & x_4 & x_5 & x_7 & x_8 & x_5 & x_6 & x_7 & x_8\\
x_7 & x_5 & x_8 & x_3 & x_7 & x_8 & x_6 & x_4 & x_2 & x_8 & x_7 & x_6 & x_8 & x_3 & x_7 & x_8 & x_8 & x_2 & x_4 & x_6 & x_8 & x_8 & x_5 & x_8 & x_5 & x_1 & x_8 & x_7 & x_8 & x_8 & x_1 & x_7\\
x_8 & x_4 & x_5 & x_8 & x_8 & x_2 & x_8 & x_6 & x_8 & x_7 & x_8 & x_1 & x_6 & x_8 & x_4 & x_7 & x_7 & x_8 & x_7 & x_8 & x_6 & x_5 & x_8 & x_2 & x_8 & x_5 & x_7 & x_8 & x_3 & x_3 & x_7 & x_1\\
x_7 & x_5 & x_8 & x_5 & x_7 & x_8 & x_5 & x_8 & x_7 & x_1 & x_2 & x_6 & x_8 & x_6 & x_7 & x_8 & x_8 & x_7 & x_8 & x_3 & x_4 & x_1 & x_3 & x_8 & x_6 & x_8 & x_4 & x_2 & x_8 & x_8 & x_8 & x_7\\
x_1 & x_8 & x_3 & x_8 & x_4 & x_7 & x_4 & x_5 & x_8 & x_7 & x_8 & x_8 & x_3 & x_8 & x_8 & x_2 & x_2 & x_8 & x_7 & x_8 & x_6 & x_5 & x_8 & x_7 & x_1 & x_6 & x_7 & x_8 & x_6 & x_5 & x_7 & x_8\\
x_5 & x_2 & x_8 & x_7 & x_6 & x_4 & x_7 & x_8 & x_5 & x_8 & x_6 & x_2 & x_8 & x_7 & x_3 & x_8 & x_8 & x_6 & x_8 & x_7 & x_8 & x_8 & x_7 & x_1 & x_7 & x_8 & x_8 & x_5 & x_1 & x_4 & x_8 & x_3\\
x_1 & x_8 & x_2 & x_8 & x_4 & x_5 & x_1 & x_7 & x_8 & x_6 & x_8 & x_8 & x_2 & x_8 & x_8 & x_3 & x_3 & x_8 & x_6 & x_8 & x_7 & x_7 & x_8 & x_6 & x_4 & x_7 & x_5 & x_8 & x_7 & x_7 & x_5 & x_8\\
x_6 & x_7 & x_8 & x_2 & x_5 & x_8 & x_7 & x_1 & x_3 & x_8 & x_6 & x_7 & x_8 & x_2 & x_6 & x_8 & x_8 & x_3 & x_4 & x_7 & x_8 & x_8 & x_7 & x_8 & x_7 & x_4 & x_8 & x_5 & x_8 & x_8 & x_1 & x_5\\
x_8 & x_1 & x_7 & x_8 & x_8 & x_3 & x_8 & x_7 & x_8 & x_6 & x_8 & x_4 & x_7 & x_8 & x_4 & x_6 & x_5 & x_8 & x_5 & x_8 & x_7 & x_7 & x_8 & x_3 & x_8 & x_7 & x_5 & x_8 & x_2 & x_2 & x_6 & x_1\\
x_8 & x_8 & x_5 & x_1 & x_8 & x_7 & x_8 & x_3 & x_1 & x_7 & x_8 & x_8 & x_6 & x_4 & x_8 & x_7 & x_7 & x_4 & x_2 & x_8 & x_5 & x_6 & x_8 & x_7 & x_8 & x_3 & x_7 & x_8 & x_6 & x_5 & x_2 & x_8\\
x_3 & x_7 & x_4 & x_7 & x_3 & x_8 & x_2 & x_8 & x_5 & x_8 & x_5 & x_7 & x_1 & x_7 & x_6 & x_1 & x_4 & x_6 & x_8 & x_7 & x_8 & x_8 & x_7 & x_8 & x_2 & x_8 & x_8 & x_6 & x_8 & x_8 & x_8 & x_5\\
x_8 & x_8 & x_7 & x_4 & x_8 & x_5 & x_8 & x_2 & x_1 & x_5 & x_8 & x_8 & x_7 & x_1 & x_8 & x_6 & x_5 & x_4 & x_3 & x_8 & x_7 & x_7 & x_8 & x_6 & x_8 & x_2 & x_6 & x_8 & x_7 & x_7 & x_3 & x_8\\
x_2 & x_6 & x_4 & x_5 & x_2 & x_8 & x_3 & x_8 & x_7 & x_8 & x_7 & x_5 & x_1 & x_6 & x_7 & x_4 & x_1 & x_7 & x_8 & x_6 & x_8 & x_8 & x_5 & x_8 & x_3 & x_8 & x_8 & x_7 & x_8 & x_8 & x_8 & x_7\\
x_8 & x_8 & x_6 & x_4 & x_8 & x_7 & x_8 & x_3 & x_4 & x_7 & x_8 & x_8 & x_5 & x_1 & x_8 & x_7 & x_7 & x_1 & x_2 & x_8 & x_6 & x_5 & x_8 & x_7 & x_8 & x_3 & x_7 & x_8 & x_5 & x_6 & x_2 & x_8\\
x_6 & x_7 & x_8 & x_7 & x_5 & x_8 & x_7 & x_8 & x_5 & x_4 & x_3 & x_7 & x_8 & x_7 & x_5 & x_8 & x_8 & x_6 & x_8 & x_2 & x_4 & x_1 & x_2 & x_8 & x_7 & x_8 & x_1 & x_3 & x_8 & x_8 & x_8 & x_6\\
x_7 & x_3 & x_8 & x_6 & x_7 & x_4 & x_5 & x_8 & x_7 & x_8 & x_7 & x_3 & x_8 & x_5 & x_2 & x_8 & x_8 & x_7 & x_8 & x_6 & x_8 & x_8 & x_5 & x_1 & x_6 & x_8 & x_8 & x_7 & x_4 & x_1 & x_8 & x_2\\
x_6 & x_2 & x_8 & x_7 & x_5 & x_1 & x_7 & x_8 & x_6 & x_8 & x_5 & x_2 & x_8 & x_7 & x_3 & x_8 & x_8 & x_5 & x_8 & x_7 & x_8 & x_8 & x_7 & x_4 & x_7 & x_8 & x_8 & x_6 & x_4 & x_1 & x_8 & x_3\\
x_2 & x_5 & x_1 & x_6 & x_2 & x_8 & x_3 & x_8 & x_7 & x_8 & x_7 & x_6 & x_4 & x_5 & x_7 & x_1 & x_4 & x_7 & x_8 & x_5 & x_8 & x_8 & x_6 & x_8 & x_3 & x_8 & x_8 & x_7 & x_8 & x_8 & x_8 & x_7\\
x_3 & x_7 & x_1 & x_7 & x_3 & x_8 & x_2 & x_8 & x_6 & x_8 & x_6 & x_7 & x_4 & x_7 & x_5 & x_4 & x_1 & x_5 & x_8 & x_7 & x_8 & x_8 & x_7 & x_8 & x_2 & x_8 & x_8 & x_5 & x_8 & x_8 & x_8 & x_6\\
x_8 & x_8 & x_7 & x_1 & x_8 & x_6 & x_8 & x_2 & x_4 & x_6 & x_8 & x_8 & x_7 & x_4 & x_8 & x_5 & x_6 & x_1 & x_3 & x_8 & x_7 & x_7 & x_8 & x_5 & x_8 & x_2 & x_5 & x_8 & x_7 & x_7 & x_3 & x_8\\
x_7 & x_6 & x_8 & x_6 & x_7 & x_8 & x_6 & x_8 & x_7 & x_4 & x_2 & x_5 & x_8 & x_5 & x_7 & x_8 & x_8 & x_7 & x_8 & x_3 & x_1 & x_4 & x_3 & x_8 & x_5 & x_8 & x_1 & x_2 & x_8 & x_8 & x_8 & x_7\\
\end{bmatrix}
\]
\end{footnotesize}

\noindent
for some $x_1,\ldots, x_8\in \Z$ corresponding to the decomposition of $G$ into eight double cosets $H_1gH_2$, consisting of $2,2,2,2,3,3,6,12$ right cosets of $H_1$, respectively.  A (nontrivial) calculation finds that
\begin{align*}
\det M = -\,&(2(x_2-x_3)^2+3(x_5-x_6)^2)^8\\
\cdot\ &(2(x_1-x_4)+(x_5+x_6-2x_7))^6\\
\cdot\ &(2(x_1+x_2+x_3+x_4)-(x_5+x_6+2x_7+4x_8))^3\\
\cdot\ &(2(x_1-x_2-x_3+x_4)-(x_5+x_6+2x_7-4x_8))^3\\
\cdot\ &(2(x_1-x_4)-3(x_5+x_6-2x_7))^2\\
\cdot\ &(2(x_1+x_2+x_3+x_4)+3(x_5+x_6+2x_7+4x_8))\\
\cdot\ &(2(x_1-x_2-x_3+x_4)+3(x_5+x_6+2x_7-4x_8)).
\end{align*}
The assignment
\[
x_1=x_2=1,\quad x_3=-1,\quad x_4=x_5=x_6=x_7=x_8=0
\]
gives $\det M = 2^{32}$, while the assignment
\[
x_5=1,\quad x_1=x_2=x_3=x_4=x_6=x_7=x_8=0
\]
gives $\det M = 3^{12}$; thus $d(H_1,H_2)=1$.
It follows that $H_1$ and $H_2$ are locally integrally equivalent, by Proposition~\ref{prop:localequiv}.
But no assignment of $x_1,\ldots,x_8\in \Z$ makes $\det M=\pm 1$; indeed,  any such assignment would require all 7 factors of $\det M$ listed above to have values in $\{\pm 1\}$, which is not possible.  Thus $H_1$ and $H_2$ are not integrally equivalent; as noted in the introduction, this negatively answers Question 2.10 of Guralnick and Weiss in \cite{GWe}.

There are infinitely many nonisomorphic variations of this example; replacing $f_1(x)$ and $f_2(x)$ with $f_1(x\sqrt{T})$ and $f_2(x\sqrt{T})$ yields polynomials with Galois group $G$ over $\Q(T)$; for almost all squarefree $a\in \Z$ the substitution $T=a$ yields nonisomorphic $K_1,K_2$ ramified at primes dividing $a$.

\subsection{A degree 96 example of solvable equivalence}\label{sec:d96}

The results of \S\ref{sec:d32} imply that any group $G$ that contains nonconjugate solvably equivalent subgroups must have order at least $32\cdot 60 = 1920$, since nonconjugate locally integrally equivalent subgroups must have index at least 32, and nonsolvable groups must have order at least 60.  A search of the small groups database shows that there are no such $G$ of order 1920 or 1980, and a search of transitive groups of degree up to 48 and order at most $48\cdot 60 = 2880$ finds no such $G$, which implies a lower bound of 2940.

An exhaustive search of transitive groups of degree up to 48 and order at most 48,000 finds transitive groups of degrees 12, 16, 20, 24, 30, 32, 36, and 40 that contain nonconjugate solvably equivalent subgroups, including examples of index 96, 192, 384, 576, 672, and 768.  The first example of index~96 occurs for
the transitive group \href{https://www.lmfdb.org/GaloisGroup/16T1654}{\texttt{16T1654}} of order 5760, which is the smallest order we found.  This group~$G$ contains five conjugacy classes of subgroups isomorphic to $A_5$, of which exactly two have representatives~$H_1$ and $H_2$ with the property that every proper subgroup of $H_1$ is also a proper subgroup of~$H_2$.  The groups~$H_1$ are thus solvably equivalent subgroups of index 96.
There are 5 double cosets $H_1gH_2$, comprised of $5,6,10,15,60$ right cosets of $H_1$, respectively; each $M\in \Hom_{\Z[G]}(\Z[H_1\backslash G],\Z[H_2\backslash G])$ can thus be viewed as a matrix in indeterminates $x_1,x_2,x_3,x_4,x_5\in \Z$, and we have
\begin{align*}
\det M = -\,&(5x_1+6x_2+10x_3+15x_4+60x_5)\\
\cdot\ &(x_1-6x_2-10x_3+3x_4+12x_5)^5\\
\cdot\ &(3x_1+2x_2-2x_3-7x_4+4x_5)^{15}\\
\cdot\ &(3x_1-2x_2+2x_3+x_4-4x_5)^{30}\\
\cdot\ &(x_1+2x_2-2x_3+3x_4-4x_5)^{45}
\end{align*}
By solving 32 systems of linear equations, one finds that no assignment of $x_1,x_2,x_3,x_4,x_5\in \Z$ makes every factor in $\det M$ equal to $\pm 1$.
Thus $H_1$ and $H_2$ are not integrally equivalent.

The regular inverse Galois problem for \texttt{16T1654} is known (it is a quotient of \href{https://www.lmfdb.org/GaloisGroup/12T277}{\texttt{12T277}}), thus there are infinitely many pairs of solvably equivalent number fields $K_1$ and $K_2$ with Galois closure $L$ that satisfy $\Gal(L/\Q)=G$, $\Gal(L/K_1)=H_1$, $\Gal(L/K_2)=H_2$.  For example, we may take $L$ as the splitting field of
\small
\[
x^{16} - 2x^{15} + 3x^{14} - 16x^{13} + 18x^{12} - 10x^{10} + 40x^9 - 39x^8 + 54x^7 + 23x^6 + 16x^5 - 140x^4 - 188x^3 - 28x^2 + 104x - 4,
\]
\normalsize
corresponding to the number field with LMFDB label \href{https://www.lmfdb.org/NumberField/16.4.711702043399998895292416.2}{\texttt{16.4.711702043399998895292416.2}}.  The field $L$ contains solvable equivalent subfields $K_1$ and $K_2$ of degree 96 that are necessarily arithmetically equivalent, locally isomorphic, and have isomorphic class groups, by Proposition~\ref{prop:solvablearithequiv}.  One can find 190 examples of \texttt{16T1654} number fields in the Kl\"uners and Malle  \href{http://galoisdb.math.uni-paderborn.de/groups/view?deg=16&num=1654}{Database of Number Fields} \cite{KM}.

\begin{remark}
In \cite[Remark~4.3]{Scott92} Scott raises several questions related to integral permutation modules that lie in the same genus, which in our setting corresponds to local integral equivalence.
The \emph{rank} of a group $G$ acting on a finite set $\Omega$ is the number of orbits of the diagonal action on $\Omega\times\Omega$.
Scott shows that if the rank of $G$ acting on $\Omega$ is 2 or 3 then local integral equivalence of $\Z[G]$-modules $\Omega$ and $\Omega'$ implies an isomorphism of $G$-sets \cite[Proposition~4.1]{Scott92}.  His example with $G=\PSL_2(29)$ proves that this does not hold when the rank is $8$.  The example in \S\ref{sec:d96} shows that this also fails to hold when the rank is $5$.
\end{remark}



\section*{Acknowledgments} 
The author would like to thank Alex Bartel for introducing him to the notion of a Gassmann triple at an LMFDB workshop at Oregon State University in 2015.  The author is also grateful to Yuval Flicker for the clarification noted in Remark~\ref{rem:flickererratum}, and to David Roe and Raymond van Bommel for their assistance in finding a particularly efficient method to compute ramification indices that is exploited in \S\ref{sec:locnonlociso}.

\bibliographystyle{amsplain}


\begin{dajauthors}
\begin{authorinfo}[avs]
  Andrew V. Sutherland\\
  Department of Mathematics\\
  Massachusetts Institute of Technology\\
  77 Massachusetts Avenue\\
  Cambridge, Massachusetts  02139\\
  USA\\
  \href{mailto:drew@math.mit.edu}{\texttt{drew@math.mit.edu}}\\
  \href{https://math.mit.edu/~drew}{\texttt{https://math.mit.edu/\raisebox{0.5ex}{\texttildelow}drew}}
\end{authorinfo}
\end{dajauthors}

\end{document}